\documentclass{llncs}
\usepackage{amsmath, amssymb}
\usepackage{url}
\usepackage{mathtools}

\def \N {\mathbb{N}}
\def \Z {\mathbb{Z}}

\def \CF {\mathcal{CF}}

\def \A {\mathcal{A}}
\def \B {\mathcal{B}}

\def \ol {\overline}
\def \ra {\rightarrow}
\def \Ra {\Rightarrow}

\def \FIM {\operatorname{FIM}}
\def \WP {\operatorname{WP}}
\def \coWP {\operatorname{coWP}}
\def \olx {\overline{x}}

\def \inv {\mathrm{inv}}
\def \rev {\mathrm{rev}}

\def \posid {{E^+}}
\def \negid {{E^-}}

\def \updown {A}
\def \downup {B}
\def \up {C}
\def \down {D}
\def \sup {E}
\def \sdown {F}
\def \mup {C'}
\def \mdown {D'}
\def \wiggle {Z}

\def \nN {\mathcal{N}}

\newtheorem{conj}{Conjecture}

\begin{document}

\title{Word Problem Languages for Free Inverse Monoids}

\author{Tara Brough
\thanks{The author was supported by the Funda\c{c}\~{a}o para a Ci\^{e}ncia e a Tecnologia (Portuguese Foundation for Science and Technology)
through an {\sc FCT} post-doctoral fellowship ({\sc SFRH}/{\sc BPD}/121469/2016) and the projects {\sc UID}/Multi/04621/2013 (CEMAT-CI\^ENCIAS) 
and {\sc UID}/{\sc MAT}/00297/2013
  (Centro de Matem\'{a}tica e Aplica\c{c}\~{o}es).
  }
  }
\institute{
Centro de Matem\'{a}tica e Aplica\c{c}\~{o}es,
Faculdade de Ci\^{e}ncias e Tecnologia,
Universidade Nova de Lisboa,
2829--516 Caparica,
Portugal\\
\email{
t.brough@fct.unl.pt
}
}

\maketitle

\begin{abstract}
This paper considers the word problem for free inverse monoids of finite rank from a 
language theory perspective.  It is shown that no free inverse monoid has context-free word problem; that the word problem of the free inverse monoid of rank $1$ 
is both $2$-context-free (an intersection of two context-free languages) and ET0L;
that the co-word problem of the free inverse monoid of rank $1$ is context-free;
and that the word problem of a free inverse monoid 
of rank greater than $1$ is not poly-context-free.

\end{abstract}

\section{Introduction}

The word problem of a finitely generated semigroup is, informally, 
the problem of 
deciding whether two words over a given finite generating set
represent the same element of the semigroup.  Although it is 
undecidable \cite{Post}, even for finitely presented 
groups \cite{Nov,Boo}, there has been much study
(especially for groups)
of word problems that are in some sense `easily' decidable, for 
example by having low space or time complexity, or being in
certain low-complexity language classes.

For groups, the obvious formalisation of the word problem is
as the set of all words over the set of generators and their
inverses representing
the identity element, since if $u$ and $v$ are words representing
the same element then $uv^{-1}$ represents the identity.  
This has been generalised to semigroups in two ways: 
the first, which we shall call the \emph{word problem} of 
a semigroup $S$ with respect to finite generating 
set $A$ is the set 
$\WP(S,A) = \{u\#v^\rev \mid u=_S v, u,v\in A^+\}$
(where $\#$ is a symbol not in $A$);
the second, the \emph{two-tape word problem} of 
$S$ with respect to $A$, is the relation
$\iota(S,A) = \{(u,v)\in A^+\times A^+ \mid u=_S v\}$.
Monoid versions of these are obtained by 
replacing $A^+$ with $A^*$.
The word problem has been studied in \cite{DG,HHOT,HOT}
and the two-tape word problem in \cite{Pfe,Bro}.

A semigroup $S$ is \emph{inverse} if for every 
$x\in S$ there is a \emph{unique} $y\in S$ 
such that $xyx=x$ and $yxy=y$.
The classes of inverse semigroups and inverse
monoids each form varieties of algebras and 
hence contain free objects.
The free inverse monoid on a set $X$ is denoted 
$\FIM(X)$; if $|X|=k$ then we also use the 
notation $\FIM_k$, and $k$ is called the \emph{rank}
of $\FIM_k$.
All results in this paper are
stated for free inverse monoids, but are equally true for 
free inverse semigroups, since in a free inverse monoid the 
only representative of the identity is the empty word $\epsilon$.

Word problems for free inverse monoids have already 
been studied from a time and space complexity perspective: 
the word problem of a free inverse monoid of 
finite rank is recognisable in linear time and in logarithmic 
space \cite{LO}.  The aim of this paper is to understand these 
word problems from a language-theoretic perspective.
All free inverse monoid word problems are context-sensitive, 
since this is equivalent to recognisability in linear space.  
Our main goal is thus
to determine in which of the many subclasses of the context-sensitive
languages the free inverse monoid word problem might lie.   
Before summarising the results, we introduce several of the 
language classes considered.
All classes mentioned here are closed under inverse generalised sequential
machine mappings, and hence the property of having word problem in any of these classes is closed under change of finite generating set.  

The non-closure of the class $\CF$ of context-free languages under
complementation and intersection \cite{HopUll}
leads naturally to the definition of
the classes of \emph{co$\CF$} and \emph{poly-$\CF$} languages, 
being respectively the classes of complements and finite intersections of 
context-free languages.  A language is called \emph{$k$-$\CF$}
if it is an intersection of $k$ context-free languages.
Groups with context-free co-word problem were studied in 
\cite{HRRT}, and groups with poly-context-free word problem
in \cite{Bro14}.
For groups, having co$\CF$ word problem is equivalent to 
the \emph{co-word problem} (the complement of the word problem, 
or abstractly the problem of deciding whether two words represent
\emph{different} elements) being context-free.  For monoids, we 
generalise this terminology on the abstract rather than technical level:
the complement of the word problem is not an algebraically interesting
language, so to define the \emph{co-word problem} of a monoid we
replace $u=_M v$ by $u\neq_M v$ (for the two-tape word problem
this is the same as taking the complement).  

\emph{Stack automata}, introduced in \cite{GGH2},
are a generalisation of 
pushdown automata that allow the contents of the stack to be examined in 
`read-only' mode.  They are a special case of the
\emph{nested stack automata} introduced slightly 
later by Aho~\cite{Aho} to recognise indexed languages.
The \emph{checking stack languages} are recognised by 
the more restricted
\emph{checking stack automata} \cite{Gre}, in which
the stack contents can only be altered 
prior to commencing reading of the input.

\emph{ET0L languages} are another subclass of indexed languages,
standardly defined by \emph{ET0L-systems}, which are essentially
finite collections of `tables' of context-free-grammar-type productions.
These operate similarly to context-free grammars except that at 
each step in a derivation, productions all from the same `table' must be
applied to \emph{every} nonterminal in the current string
(each table is required to have productions from every 
nonterminal, though these of course may be trivial).  
The more restricted \emph{EDT0L languages} have the further requirement
that in each table of productions there be only one production from each 
nonterminal.
An automaton model for ET0L languages was given in \cite{vLee}:
it consists of a checking stack with attached push-down stack, 
operating in such a way that the pointers of the two stacks move together.
See \cite{SR} for further information on ET0L languages and their 
many relatives.

In the rank $1$ case our goal is achieved fairly comprehensively, 
with both types of word problem for $\FIM_1$
being shown to be $2$-$\CF$ (but not context-free), 
co-$\CF$ and a checking stack language 
(and hence ET0L).  As far as the author is aware, this 
is the first known example of a semigroup with ET0L
but not context-free word problem.  This result is particularly 
interesting because of the long-standing open problem
of whether the indexed languages
-- of which the ET0L languages form a subclass -- give 
any additional power over context-free languages for
recognising word problems of groups \cite{LR,GS}.
In higher ranks we show that 
$\WP(\FIM_k)$ for $k\geq 2$ is not poly-$\CF$.
We conjecture that the same is true for $\iota(\FIM_k)$,
and that neither version of the word problem is co$\CF$
or indexed except in rank~$1$.



\section{Background}


\subsection{Free inverse monoids}

Recall that a monoid $M$ is \emph{inverse} if for every 
$x\in M$ there is a unique $y\in M$
such that $xyx = x$ and $yxy = y$.  
The element $y$ is called the \emph{inverse of $x$}
and is usually denoted $x^{-1}$.  
In this paper we will also often use the notation 
$\olx$ for the inverse of $x$.
Given a set $X$, we use the notation $X^{-1}$
for a set $\{\olx \mid x\in X\}$ of formal inverses for $X$,
and $X^\pm$ for $X\cup X^{-1}$.
For an element $x\in X^\pm$, if $x\in X$ then
$x^{-1}=\olx$, while if $x=\ol{y}$ for $y\in X$ then 
$x^{-1}=y$.  We can extend this to define the inverse 
of a word $w = w_1\ldots w_n$ with $w_i\in X^\pm$
by $w^\inv = w_n^{-1}\ldots w_1^{-1}$.

For any set $X$, the \emph{free inverse monoid}
$\FIM(X)$ on $X$ exists and is given by the monoid presentation
\[ \FIM(X) = \langle X^\pm \mid u = uu^\inv u, uu^\inv = u^\inv u
\; (u\in (X^\pm)^*)\rangle.\]

This presentation is not particularly useful for working 
with the word problem of free inverse monoids.
A much more powerful tool is given by \emph{Munn trees}
\cite{Mun}.
These are certain labelled directed finite trees that stand in 
one-to-one correspondence with the elements of $\FIM(X)$,
and such that the product of two elements can easily be 
computed using their corresponding trees.
To obtain the Munn tree for an element $m\in \FIM(X)$,
we use the Cayley graph $\mathcal{G}(F(X),X)$ 
of the free group $F(X)$.
This can be viewed as a labelled directed tree with 
$|X|$ edges labelled by the elements of $X$ entering 
and leaving at each vertex.  (The Cayley graph also 
technically has its vertices labelled by the elements of $G$,
but these are not needed for our purposes.)  
Given any word $w\in (X^\pm)^*$ representing $m$,
we choose any vertex of $\mathcal{G}(F(X),X)$ as the 
start vertex and label it $\alpha$.  From vertex $\alpha$,
we then trace out the path defined by reading $w$
from left to right, where for $x\in X$, we follow the edge
labelled $x$ leading \emph{from} the current vertex upon 
reading $x$, and follow the edge labelled $x$ leading 
\emph{to} the current vertex upon reading $\olx$.
We mark the final vertex of the path thus traced as 
$\omega$, and remove all edges not traversed during 
reading of $w$.  The result is the \emph{Munn tree} of 
$w$, and the free inverse monoid relations ensure that 
two words produce the same Munn tree if and only if they 
represent the same element of $\FIM(X)$.

To multiply two Munn trees, simply attach the start vertex
of the second tree to the end vertex of the first tree, and 
identify any edges with the same label and direction 
issuing from or entering the same vertex.
From this it can be seen that the \emph{idempotents}
(elements $x$ such that $x^2=x$) in $\FIM(X)$ are those elements
whose Munn trees have $\alpha=\omega$, 
and that these elements commute.

\subsection{Word problems of inverse monoids}

Two notions of word problem for inverse monoids 
will occur throughout this paper.  
For an inverse monoid $M$ with finite generating set $X$, 
the \emph{word problem} of $M$ with respect to $X$ is the set 
\[ \WP(M,X) = \{ u\#v^\inv \mid u,v\in (X^\pm)^*, u=_M v\},\]
while the \emph{two-tape word problem} of $M$ with 
respect to $X$ is
\[\iota(M,X) = \{ (u,v)\in (X^\pm)^*\times (X^\pm)^* \mid u=_M v\}. \]
If the generating set $X$ is irrelevant, we may use the notation
$\WP(M)$ or $\iota(M)$.

Each of these notions generalises the definition
of the group word problem $W(G,X)$ as the set of all words over 
$X^\pm$ representing the identity.  
If $M$ is a group, then $W(G,X)$ and $\WP(G,X)$ are obtained 
from each other by very simple operations (deletion or insertion 
of a single $\#$), and so membership in any `reasonable' language
class will not depend on whether we consider the group or inverse 
monoid word problem.
For the two-tape word problem the generalisation is of a more 
algebraic nature: $\iota(M,X)$ and $W(G,X)$ are each the lift 
to $(X^\pm)^*$ of the natural homomorphism from the free 
inverse monoid (respectively free group) on $X$ to $M$
(respectively $G$).  The kernel of a group homomorphism is a 
set, while the kernel of a semigroup homomorphism is a relation.

The word problem for semigroups in general has been studied 
in \cite{HHOT}, where it is defined as the set of words 
$u\#v^\rev$ with $u$ and $v$ representing the same element.
For inverse monoids, this is equivalent to the 
word problem considered here, since $u\#v^\inv$ is obtained 
from $u\#v^\rev$ by simply replacing every symbol after the $\#$ 
by its inverse.  This operation can be viewed as an inverse generalised
sequential machine mapping, and thus all classes of languages we 
consider are closed under it (and hence all results in this paper hold
for the definition in \cite{HHOT} as well).  
Note that it is still essential to 
include the `dividing symbol' $\#$: as an example, if $F = \FIM(X)$
and $x\in X$, then $x\#\olx\in \WP(F,X)$, but 
$x\olx\#\notin \WP(F,X)$.


\section{The rank $1$ case}\label{rank1}

Since the free group of rank $1$ is isomorphic to $(\Z,+)$,
Munn trees in the rank $1$ case can be viewed as intervals of integers containing 
zero (the starting point $\alpha$), with a marked point ($\omega$).  
This allows elements of $\FIM_1$ to be represented 
by a $3$-tuple of integers $(-l,n,m)$ with $l,n\in \N_0$ and $-l\leq m\leq n$, where 
$[-l,n]$ is the interval spanned by the Munn tree and $m$ is the marked point.
Multiplication in this representation of $\FIM_1$ is given by
\[ (-l,n,m) (-l',n',m') = (-l\land (m-l'), n\lor (m+n'), m+m'). \]
Equipped with this model of $\FIM_1$, we can determine that free inverse monoids
never have context-free word problem.

\begin{theorem}\label{notCF}
For any $k\in \N$, neither $\WP(\FIM_k)$ nor $\iota(\FIM_k)$ is context-free.
\end{theorem}
\begin{proof}
Suppose that $\WP(\FIM_k,X)$ is context-free ($X$ any finite generating set 
of $\FIM_k$).  Then for any $x\in X$, the language
$L := \WP(\FIM_k,X)\cap x^*\olx^*x^*\#\olx^*$ is also context-free.
For $n\in \N$, let $w_n = x^n \olx^n x^n \# \olx^n$, which is in $L$ for all $n\in \N$.
For $n$ greater than the pumping length $p$ of $L$, we can express $w_n$ in the form
$uvwyz$ such that $|vy|\geq 1$, $|vwy|\leq p$, and the strings $v,y$ can simultaneously
be `pumped'.  Thus there must exist $i,j\in \N_0$, not both zero, such that all strings of one
of the following three forms must be in $L$ for $m\geq -1$:
\[
x^{n+im} \olx^{n+jm} x^n \# \olx^n,  \quad
x^n \olx^{n+im} x^{n+jm} \# \olx^n,   \quad
x^n \olx^n x^{n+im} \# \olx^{n+jm}.
\]

However, in all cases, some words of the given form are not in $L$:
\[
\begin{tabular}{l | c}
Word form &  Not in $L$ for\\
\hline
$x^{n+im} \olx^{n+jm} x^n \# \olx^n$ & $(i\neq 0 \land m\geq 1) \lor (i=0 \land j\neq 0 \land m\geq 1)$\\
$x^n \olx^{n+im} x^{n+jm} \# \olx^n$ & as above\\
$x^n \olx^n x^{n+im} \# \olx^{n+jm}$ & $( j\neq 0 \land m=-1) \lor (j=0 \land i\neq 0 \land m\geq 1)$
\end{tabular}
\]
Hence $L$, and therefore $\WP(\FIM_k,X)$, is not context-free.
The proof for $\iota(\FIM_k,X)$ is similar, using the pumping lemma
on $(x^n \olx^n x^n, x^n)$ for sufficiently large $n$.  
\qed
\end{proof}

For the remainder of this section, let $\FIM_1$ be generated by 
$X = \{x\}$ and let $Y=X^\pm=\{x,\olx\}$.  
For $w\in Y^*$, denote the image of $w$ in $\FIM_1$ by $\hat{w}$.
We define functions $\lambda, \nu$ and $\mu$
from $Y^*$ to $\Z$ by setting $(-\lambda(w), \nu(w), \mu(w) )= \hat{w}$.
It will often be helpful to regard words in $Y^*$ as paths in the integers
starting at $0$, 
with $x$ representing a step in the positive direction and $\olx$ a step in the 
negative direction.  We will refer to and visualise these directions as
right (positive) and left (negative).
Thus for $w\in Y^*$ the path traced out by $w$ has rightmost point 
$\nu(w)$, leftmost point $-\lambda(w)$ and endpoint $\mu(w)$.

The idempotents in $\FIM_1$ are the elements
$(-l,n,0)$ for $l,n\in \N_0$.  We define the set 
of \emph{positive idempotents} 
$\posid = \{ (0,n,0) \mid n\in \N_0 \}$
and similarly the set of \emph{negative idempotents}
$\negid = \{ (-l,0,0) \mid l\in \N_0 \}$ in $\FIM_1$.
(Note that in these definitions, the identity $(0,0,0)$ is
counted as both a positive and a negative idempotent.)
Grammars for the sets of positive and negative idempotents form an 
important building block in Theorem~\ref{2CF} (as well as in the 
ET0L grammar mentioned following Corollary~\ref{ET0L}).

\begin{lemma}\label{downzero}
Let $Y = \{x,\olx\}$ and $L_{\posid} = \{w\in Y^* \mid \hat{w} \in \posid\}$.
Then $L_\posid$ is generated by the context-free grammar 
$\Gamma^+ = (\{S\},Y,P^+,S)$ with $P^+$ consisting of 
productions $P_1: S\rightarrow SS$, $P_2: S\rightarrow x S \olx$
and $P_3: S\rightarrow \varepsilon$.
Similarly, $L_\negid := \{w\in Y^* \mid \hat{w}\in \negid\}$
is generated by the context-free grammar 
$\Gamma^- = (\{S\}, Y, P^-, S)$ with $P^-$ the same as 
$P^+$ except that $P_2$ is replaced by $P'_2: S\rightarrow \olx S x$.
\end{lemma}
\begin{proof}
We can view $L_\posid$ as the language of all paths starting and ending at 
$0$ and never crossing to the left of $0$.  
Concatenating two such paths gives another such path, so we have 
$(L_\posid)^* = L_\posid$.
Let $L$ be the language of all paths in $Y^*$ that start and 
end at $0$ without visiting $0$ in between.  Then $L_\posid = L^*$
and $w\in Y^*$ is in $L$ if and only if either $w=\varepsilon$ or there exists 
$v\in L_\posid$ such that $w = xv\olx$.  
That is, $L_\posid = (xL_\posid\olx)^*$.

Let $M$ be the language generated by $\Gamma^+$.
We show by induction on the length of words that 
$L_\posid = M$.
Note that for any $w$ in $L_\posid$ or $M$ we have $|w|_x = |w|_{\olx}$,
so both languages consist of words of even length.
To begin with, $L_\posid\cap Y^0 = M\cap Y^0 = \{\varepsilon\}$.
Now suppose that $L_\posid\cap Y^{2i} = M\cap Y^{2i}$
for all $i< n$.
For $w\in Y^{2n}$, we have $w\in L_\posid$ if and only if either
$w = w_1 w_2$  or $w = xv\olx$ for 
some $w_1,w_2,v\in L_\posid$.
By induction, this occurs iff $w_1,w_2\in M$ respectively $v\in M$, iff
$S\ra SS\Ra w_1w_2$ respectively $S\ra xS\olx \Ra xv\olx$ in $\Gamma^+$,
iff $w\in M$.
Hence $L_\posid = M$.
The language $L_\negid$ is the reverse of $L_\posid$,
and the grammars $\Gamma^+$ and $\Gamma^-$ are 
the reverse of one another, hence $L_\negid$ is generated
by $\Gamma^-$.
\qed
\end{proof}

A word $u\#v^\inv$ with $u,v\in Y^*$ is in $\WP(\FIM_1,X)$ 
if and only if it traces out a path starting and ending at $0$ which 
reaches its rightmost and leftmost points each at least once before and 
at least once after the $\#$.  If the minimum or maximum is achieved at the 
end of $u$, this counts as being achieved both before and after $\#$.  
(The path must end at $0$ because if 
$\hat{u} = \hat{v}$ then $\hat{u}(\hat{v})^{-1}$ is an idempotent.) 
We now show that although the word problem of $\FIM_1$ is not 
context-free, it can be expressed as an intersection of two 
context-free languages.  

\begin{theorem}\label{2CF}
$\WP(\FIM_1)$  and $\iota(\FIM_1)$ are both $2$-$\CF$.
\end{theorem}
\begin{proof}
Let $X=\{x\}$, $Y = \{x,\olx\}$ and $L = \WP(\FIM_1,X)$.
We can express $L$ as the intersection of the following two languages:
\[ L_\nu = \{u\#v^\inv \mid u,v\in Y^*, \nu(u) = \nu(v) \land \mu(u) = \mu(v) \}\]
and 
\[ L_\lambda = \{u\#v^\inv \mid u,v\in Y^*, \lambda(u)=\lambda(v) \land \mu(u)=\mu(v) \}.\]
We will show that $L_\nu$ and $L_\lambda$ are each context-free 
and hence $L$ is $2$-$\CF$.
Since $L_\lambda$ is simply the reverse of $L_\nu$, 
it suffices to prove that $L_\nu$ is context-free.

Let $\Gamma_\nu = (V,\Sigma,P,S)$ be the context-free grammar with 
nonterminals $V = \{S,T,Z,Z'\}$, terminals $\Sigma = \{x,\olx,\#\}$ and 
productions $P$ as follows:
\begin{align*}
S&\rightarrow ZSZ \mid xS\olx \mid T\\
T&\rightarrow ZTZ \mid \olx T x \mid \#\\
Z&\rightarrow ZZ \mid \olx Z x \mid \varepsilon.
\end{align*}

Any derivation in $\Gamma_\nu$ can be expressed as
\begin{eqnarray}\label{deriv}
S\Ra \alpha S \beta \ra \alpha T \beta \Ra u_1 u_2 \#v_2 v_1,
\end{eqnarray}
where $\alpha \Ra u_1$, $\beta\Ra v_1$ and $T\Ra u_2\# v_2$.

For any $\alpha'\in \{Z,x\}^*$ and $\beta'\in \{Z,\olx\}^*$ 
with $|\alpha'|_x = |\beta'|_{\olx}$, there is a partial derivation
in $\Gamma_\nu$, not involving the production $S\ra T$,
from $S$ to $\alpha' S \beta'$.  Conversely, any partial 
derivation from $S$ not involving $S\ra T$ results in a 
string $\alpha S \beta$ in which $\alpha$ and $\beta$
can be derived from some such $\alpha'$ and $\beta'$
respectively.

Let $\alpha\in \{Z,x\}^*$ and $w\in Y^*$ with 
$\alpha\Rightarrow^* w$.
By Lemma~\ref{downzero}, the subwords of $w$
produced from instances of $Z$ in $\alpha$ evaluate
to negative idempotents, and so have no effect on 
$\nu(w)$ or $\mu(w)$, whereas each $x$ in $\alpha$
increases both $\nu(w)$ and $\mu(w)$ by $1$.
Hence $\nu(w) = \mu(w) = |\alpha|_x$.
Thus a pair of words $u_1$ and $v_1$ can appear in the derivation 
(\ref{deriv}) if and only if $\nu(u_1) = \mu(u_1) = \mu(v_1^\inv) = \mu(v_1^\inv)$. 
Similarly, it can be shown that $T\Ra u_2\# v_2$ if and only if 
$\nu(u_2) = \nu(v_2^\inv) = 0$ and $\mu(u_2) = \mu(v_2^\inv)\leq 0$.  

Hence $S\Ra u\# v$ if and only if we can write 
$u = u_1u_2$ and $v = v_2v_1$ such that
there exist $l_1,l_2,l_1',l_2',m,n\in \N_0$ with
\begin{align*}
u_1 =_{\FIM_1} (-l_1,n,n) & \hspace{1cm} u_2   =_{\FIM_1} (-l_2,-m,0)\\
v_1^\inv =_{\FIM_1} (-l_1',n,n)  & \hspace{1cm} v_2^\inv =_{\FIM_1}  (-l_2',-m,0).
\end{align*}
If $u\# v\in L_\nu$, then we can express $u$ and $v$ in this way 
by setting $u_1$ to be the shortest prefix of $u$ such that 
$\nu(u_1) = \nu(u)$ and $v_1^\inv$ the shortest prefix of 
$v^\inv$ such that $\nu(v_1^\inv)=\nu(v^\inv)$.  
Conversely, supposing we can express $u$ and $v$ in this way, we have
\begin{align*}
u &=_{\FIM_1} (-l_1,n,n)(-l_2,n,-m) = (-i,n,n-m)\\
v^\inv &=_{\FIM_1} (-l_1',n,n)(-l_2',0,-m) = (-j,n,n-m)
\end{align*}
for some $i,j\in \N_0$.  That is, $u\#v\in L_\nu$.

Hence $L_\nu$ is generated by $\Gamma_\nu$ and is context-free, 
and therefore $L_\nu^\rev = L_\lambda$ is also context-free.  
Thus $\WP(\FIM_1,X) = L_\nu\cap L_\lambda$ is $2$-$\CF$.

For variety, we give an automaton proof for $\iota(\FIM_1,X)$.
Define sublanguages $L^\iota_\nu$ and $L^\iota_\lambda$ of 
$Y^*\times Y^*$ analogously to $L_\nu$ and $L_\lambda$.
Let $x_1 = (x,\epsilon)$, $x_2 = (\epsilon,x)$ and define 
$\olx_1$, $\olx_2$ similarly.
Reading $x_i$ or $\olx_i$ means that we read an $x$ or $\olx$ 
from the $i$-th tape and nothing from the other tape.
Define a pushdown automaton $\A_\nu$ with states $q_0,q_1$
by the following transitions for $i=1,2$ ($Z$ is the bottom-of-stack symbol):
\begin{align*}
(q_0,Z,(x,x)) &\mapsto (q_0,Z)&			(q_1,Z,(\olx,\olx)) &\mapsto (q_1,Z)\\
(q_0,Z,\olx_i) &\mapsto (q_0,Y_iZ)&		(q_1,Z,\olx_i) &\mapsto (q_1,Y_iZ)\\
(q_0,Y_i,\olx_i) &\mapsto (q_0,Y_iY_i)&		(q_1,Y_i,\olx_i) &\mapsto (q_1,Y_iY_i)\\
(q_0,Y_i,x_i) &\mapsto (q_0,\epsilon)&	(q_1,Y_i,x_i) &\mapsto (q_1,\epsilon).\\
(q_0, Z, \epsilon) &\mapsto (q_1,Z)
\end{align*}
The language $\A_\nu$ accepts by empty stack consists of all pairs $(u,v)$ where 
$u,v\in (E^- x)^n (E^- \olx)^k E^-$ for some $n,k\in \N_0$, which is precisely
the language~$L^\iota_\nu$.  Switching the roles of $x$ and $\olx$ in $\A_\nu$
gives rise to a pushdown automaton $\A_\lambda$ accepting $L^\iota_\lambda$.
Hence $\iota(\FIM_1,X)$ is also $2$-$\CF$.
\qed
\end{proof}


\begin{theorem}\label{coCF}
Both versions of the co-word problem of $\FIM_1$ are 
context-free.
\end{theorem}
\begin{proof}
Let $K = \coWP(\FIM_1,X) = \{u\#v^\inv \mid u,v\in Y^*, u\neq_{\FIM_1} v\}$.  
A word $w =u\#v$ with $u,v\in Y^*$ is in 
$K$ if and only if the path traced out by $w$ starting at $0$ either does 
not end at $0$, or its minimum or maximum value is not achieved both
before and after $\#$ (recall that this includes not being achieved at
the end of $u$).  
Thus a context-free grammar for $K$ with start symbol $S$ is given by the 
following productions:
\begin{align*}
S&\ra M \mid U\mid D   &  U&\ra ZxU\olx Z \mid xE\olx Z\# \mid \#ZxE\olx Z \\
M&\ra ExA\mid E\olx B   &  D&\ra Z' \olx DxZ' \mid \olx ExZ'\#\ \mid \#Z'\olx Ex Z'\\  
A&\ra xA \mid EAE\mid \epsilon & Z&\ra ZZ\mid \olx Zx\mid \epsilon\\
B&\ra \olx B \mid EBE\mid \epsilon & Z&'\ra Z'Z'\mid xZ'\olx\mid \epsilon.\\
E&\ra ZE\mid Z'E\mid \epsilon &&
\end{align*}
$M$ generates all $u\#v^\inv$ with $\mu(u)\neq \mu(v)$;
$U$ generates all $u\#v^\inv$ with $uv^\inv$ idempotent but 
$\nu(u)\neq \nu(v)$, and $D$ does the same as $U$ but for $\lambda$
instead of $\nu$.

The two-tape co-word problem of $\FIM_1$ with respect to $X$ is
the language $M = \{(u,v)\in Y^*\times Y^* \mid u\neq_{\FIM_1} v\}$.
A pushdown automaton recognising $M$ can be expressed as 
the union of automata $\B_\mu$, $B_\nu$, $B_\lambda$.
The automaton $\B_\mu$ checks that 
$|u|_x - |u|_{\olx} \neq |v|_x - |v|_{\olx}$
for input $(u,v)$, and thus accepts all pairs with $\mu(u)\neq \mu(v)$.
The automaton $\B_\nu$ has states $q_0,q_1,q_2,f$,
with $f$ being the unique final state, input symbols 
$x_i$, $\olx_i$ (as in the proof of Theorem~\ref{2CF}) and transitions:
\begin{align*}
(q_0,x_1,Z)&\ra (q_0,XZ)  & (q_0,\epsilon,*)&\ra (q_1,*)
&(q_2,x_2,Z)&\ra (f,Z)\\
(q_0,x_1,X)&\ra (q_0,XX)  & (q_1,\epsilon,Y)&\ra (q_1,\epsilon)
&(q_2,x_2,X)&\ra (q_2,\epsilon)\\
(q_0,x_1,Y)&\ra (q_0,\epsilon) & (q_1,x_2,X)&\ra (q_2,\epsilon)
&(q_2,x_2,Y)&\ra (q_2,\epsilon)\\
(q_0,\olx_1,*)&\ra (q_0,Y*)  &  (q_1,\olx_2,X)&\ra (q_2,YX)
&(q_2,\olx_2,*)&\ra (q_2,Y*)\\
&&&&(q_2,\epsilon,X)&\ra (f,X),
\end{align*}
where $Z$ is the bottom-of-stack marker and $*$ 
denotes any stack symbol ($X,Y,Z$).
In state $q_0$, $X^{\nu(u)} Y^{\nu(u)-\mu(u)}$ is placed on the stack.
State $q_1$ removes all $Y$'s.  State $q_2$ then checks 
$\nu(v)$ against $\nu(u)$, moving to the final state $f$ if we
either find that $\nu(v)>\nu(u)$ or nondeterministically if 
$\nu(v')<\nu(u)$ for the prefix $v'$ of $v$ read so far, in which 
case $\B_\nu$ accepts if there is no further input.  
Thus $\B_\nu$ accepts the language of all $(u,v)$ 
with $\nu(u)\neq \nu(v)$.
The automaton $\B_\lambda$ is obtained by swapping
the roles of $x_i$ and $\olx_i$ in $\B_\nu$, and accepts
$(u,v)$ with $\lambda(u)\neq \lambda(v)$.
\qed
\end{proof}

Given the model of elements of $\FIM_1$ as marked intervals in $\Z$,
stack automata provide possibly the most natural class of automata to 
consider as acceptors of its word problem.  It turns out that it suffices
to use a checking stack automaton.

\begin{theorem}\label{stack}
$\WP(\FIM_1)$  and $\iota(\FIM_1)$ are each recognised by a checking 
stack automaton.
\end{theorem}
\begin{proof}
The idea of the checking stack automaton for $\WP(\FIM_1)$ is to use the stack 
contents as a model for an interval of integers $[-l,n]$ 
(chosen nondeterministically before beginning to read the input), 
and check for input $u\#v^\inv$ whether $\hat{u} = \hat{v} = (-l,n,m)$ 
for some $m\in [-l,n]$.
Following the set-up phase, the stack contents will always be of the form
$L^l O R^n$ for some $l,n\in \N_0$, with the leftmost symbol $L$ or $O$
being marked with a superscript $^-$, and the rightmost symbol $O$ or $R$
being marked with a superscript $^+$.  Such a string represents a guess 
that the input string $u\# v^\inv$ will have $\lambda(u) = \lambda(v) = l$ 
and $\nu(u) = \nu(v) = n$.  
For $\alpha\in L^* O R^*$, we denote the string 
`$\alpha$ with marked endpoints' by $[\alpha]$.  For example, 
$[LLORR] = L^-LORR^+$ and $[O] = O^\pm$.
Before beginning to consume the input, the stack marker is moved to 
$O^{(+,-,\pm)}$.

During the checking phase, the automaton $\A$ moves up and down the 
stack, tracing out the path given by $u\#v^\inv$, accepting if and only
if three conditions are satisfied:
(i)
both the left and right endpoints of $[\alpha]$ are reached at least 
once before and after the $\#$;
(ii)
the automaton never attempts to move beyond the endpoints of $[\alpha]$; and
(iii)
the automaton ends with the stack marker pointing at $O^{(+,-,\pm)}$.
If during the set-up phase the string $[L^l O R^n]$ was placed on the stack, 
then the set of words accepted by $\A$ following that particular set-up 
will be the language
$\{ u\#v^\inv \mid u,v\in Y^*, \lambda(u) = \lambda(v) = l, \nu(u) = \nu(v) = n,
\mu(u) = \nu(v)\}$.

More formally, the checking transitions of $\A$ are described as follows,
using states $\{q_i, q_i^+, q_i^-, q_i^*,f \mid i=1,2\}$, with $f$ the unique
final state.
Following the setup phase (which can be achieved non-deterministically 
using two states), $\A$ is in state $q_1$, with stack contents 
$[\alpha]$ for some string $\alpha\in L^*OR^*$, and the stack marker 
pointing at the symbol corresponding to $O$ in $[\alpha]$.
The symbol $\$$ is an end-of-input marker, standardly included 
in the definition of stack automata.
Let $\Delta^- = \{O^-, R^-\}$ and $\Delta^+ = \{O^+, L^+\}$. 
The left-hand side of a transition represents the current 
automaton configuration (state,input,stack symbol). 
The right-hand side has first component the state to 
be moved to, and second component the direction in which
to move the stack marker (with $-$ denoting no change).
The full set of stack symbols is 
$\Gamma = \{L^{(+,-)}, O^{(+,-,\pm)}, R^{(+,-)}\}$.
For $i=1,2$:
\begin{tabular}{llll}
$(q_1,\#,O^\pm)\mapsto (q_2^*,-)$, &&\\
$(q_i,x,C)\mapsto (q_i,\uparrow)$, 
&$\mathrlap{C\in \{L,O,R\}}$ &\\
$(q_i,\olx,C)\mapsto (q_i,\downarrow)$, 
&$\mathrlap{C\in \{L,O,R\}}$ &\\
$(q_i,x,C)\mapsto (q_i^-,\uparrow)$, 
&$C\in \Delta^-$
& \; $(q_i^*,x,C)\mapsto (q_i^*,\uparrow)$, 
&$C\notin \{R^+, O^+, O^\pm\}$\\
$(q_i,\olx,C)\mapsto (q_i^+,\downarrow)$, 
&$C\in \Delta^+$ 
& \; $(q_i^*,\olx,C)\mapsto (q_i^*,\downarrow)$, 
&$C\notin \{L^-, O^-, O^\pm\}$\\
$(q_i^+,x,C)\mapsto (q_i^*,\uparrow)$, 
&$C\in \Delta^-$ 
& \; $(q_1^*,\#,C)\mapsto (q_2,-)$,  & $C\in \Gamma$\\
$(q_i^-,\olx,C)\mapsto (q_i^*,\downarrow)$, 
&$C\in \Delta^+$ 
& \; $(q_2^*,\$,C)\mapsto (f,-)$,  & $C\in \{O, O^-, O^+, O^\pm\}$.\\
\end{tabular}\\

Note that these transitions involve no push or pop operations,
so $\A$ is a checking stack automaton.
Now assume that $\A$ has reached the reading phase, with 
stack contents $[L^lOR^n]$ for some $l,n\in \N_0$.  Let $L_{(l,n)}$
denote the language of all words accepted by $\A$ from this configuration.
The case $(l,n) = (0,0)$ is degenerate, since $[O] = O^\pm$ and the 
only path from $q_0$ to $f$ in this case is on input $\#\$$, which is 
exactly as desired since the empty word is the only representative of
the identity $(0,0,0)$ in $Y^*$.  Henceforth we will assume 
at least one of $l$ or $n$ is non-zero.

With few exceptions, the automaton moves up the stack on input 
$x$ and down on $\olx$.  The exceptions are when this would otherwise 
result in moving beyond the top or bottom of the stack.  In these cases 
there are no transitions defined and so the automaton fails.
Thus for $w\in Y^*$ the stack marker traces out the path given by 
$w$, provided this path remains within the interval $[-l,n]$.

When the automaton is in state $q_i$ and has reached the top of the stack
(indicated by a symbol in $\Delta^+$), on the next input it either fails (on $x$) 
or moves to state $q_i^+$ (on $\olx$).  Similarly, after reaching the bottom of 
the stack (symbols in $\Delta^-$), the automaton either fails (on $\olx$) or moves 
to $q_i^-$ (on $x$).  Following either of these events, the automaton will move
to state $q_i^*$ after reaching the opposite end of the stack, provided it does not 
fail.  Thus being in state $q_0^*$ indicates that the automaton has read some 
$u'\in Y^*$ with $\lambda(u') = l$ and $\nu(u') = n$.

The only transition on the symbol $\#$ is from state $q_0^*$ to 
$q_1$ (regardless of stack symbol),
and the only transition on $\$$ is from $q_1^*$ to the final state $f$ 
and requires the automaton to be pointing at $O^{(+,-)}$
(both of these transitions leave the stack unchanged). 
Hence $L_{(l,n)}$ contains exactly those words in $Y^*\#Y^*\$$
which trace out a path in $[-l,n]$ starting and ending at $O^{(+,-)}$ which 
visit the top and bottom of the stack each at least once before and after 
the $\#$; that is, $L_{(l,n)}$ consists of all $u\#v^\inv\$$ such that 
$u\#v^\inv$ is in $\WP(\FIM_1,X)$ and $\lambda(u)=l$, $\nu(u)=n$,
as desired.
Since the language accepted by $\A$ is $\bigcup_{l,n\in \N_0} L_{(l,n)}$,
we conclude that $\A$ accepts $\WP(\FIM_1,X)$.

To recognise $\iota(\FIM_1,X)$, we make a few small modifications to $\A$:
in the setup phase, we additionally mark some symbol of the stack 
contents $[\alpha]$ to denote a guess as to the location of $\mu(u)=\mu(v)$
(where the input is $(u\$,v\$)$).  
In states $q_i,q_i^+,q_i^-,q_i^*$, we read from the $i$-th tape ($i=1,2$).  
On reaching the end symbol $\$$ on each tape, we only proceed if the 
stack marker is pointing at the marked symbol.  
We introduce an intermediate state between $q_1^*$ and $q_2$ which 
returns the stack marker to the symbol corresponding to $O$ in $\alpha$.
In all other respects, the automaton behaves the same as $\A$.
Thus the stack contents of the modified automaton represent an element 
$(-l,n,m)$ of $\FIM_1$, and with these stack contents the automaton 
accepts all $(u,v)$ such that $u$ and $v$ both evaluate to $(-l,n,m)$.
Evaluating over all possible stack contents yields $\iota(\FIM_1,X)$.
\qed
\end{proof}

Note that the classes of $2$-$\CF$ and checking stack languages are 
incomparable.  The language $\{ww\mid w\in A^*\}$ for $|A|\geq 2$
is not poly-$\CF$
(an easy application of \cite[Theorem~3.9]{Bro14}), but is accepted 
by a checking stack automaton that starts by putting a word
$w$ on the stack and then checks whether the input is $ww$.
The language $\{(ab^n)^n\mid n\in \N\}$ is not even indexed
\cite[Theorem~5.3]{HopUll}, but is $3$-$\CF$.

Since E(D)T0L languages have been shown to describe various
languages arising in group theory \cite{CDE,CEF} (but not word problems),
it is worth noting the following.

\begin{corollary}\label{ET0L}
$\WP(\FIM_1)$ and $\iota(\FIM_1)$ are both ET0L.
\end{corollary}
\begin{proof}
This follows from Theorem~\ref{stack} and the fact that the class of 
checking stack languages is contained in the ET0L languages
\cite{vLee}.  
\qed
\end{proof}

The author has constructed nondeterministic ET0L grammars 
for $\WP(\FIM_1,X)$ and $\iota(\FIM_1,X)$ with 
$9$ tables and $11$ nonterminals.  
The nondeterminism arises from the fact that for any word 
$w\in Y^*$, we may insert idempotents arbitrarily at any point 
in $w$ without changing the element represented, provided that
these idempotents are not `large' enough to
change the value of $\nu(w)$ or $\lambda(w)$.

\begin{conj}
Neither $\WP(\FIM_1)$ nor $\iota(\FIM_1)$ is EDT0L.
\end{conj}


\section{Rank greater than $1$}

The word problem for inverse monoids in higher ranks
is more complex from a language theory perspective.

\begin{lemma}\label{notpolycf}
For any $k\geq 3$, $\WP(\FIM_k)$ is not $(k-2)$-$\CF$.
\end{lemma}
\begin{proof}
For any $k\geq 2$, let $X_k = \{x_1,\ldots,x_k,\olx_1,\ldots,\olx_k\}$ 
and let 
\[L_k = \{x_1^{m_1} \olx_1^{m_1} \ldots x_k^{m_k} \olx_k^{m_k} \#
x_1^{n_1} \olx_1^{n_1} \ldots x_k^{n_k} \olx_k^{n_k} \mid m_i,n_i\in \N_0\}.\]
Let $W_k = \WP(\FIM_k,X_k)$.  Then $W_k$ consists of all those words in $L_k$
with $m_i = n_i$ for all $i$ (since idempotents in $\FIM_k$ commute).  
By \cite[Theorem~3.12]{Bro14}, $W_k$ is not $(k-1)$-context-free
\footnote{
The language $L^{(2,k)}$ in the referenced result is not precisely
$W_k$, but the associated set of integer tuples differs from that 
associated to $W_k$ only by a constant (arising from the symbol $\#$),
which does not affect stratification properties and therefore does not 
affect the property of not being $(k-1)$-$\CF$.
}.
Since $W_k$ is the intersection of $\WP(\FIM_k,X_k)$ with the context-free language
$L_k$, this implies that $\WP(\FIM_k,X_k)$ is not context-free. 
\qed
\end{proof}

Since for $k\geq 2$ $\FIM_k$ contains submonoids isomorphic to $\FIM_n$ for all $n$,
the following theorem is immediate from Lemma~\ref{notpolycf}.

\begin{theorem}
For any $k\geq 2$, $\WP(\FIM_k)$ is not poly-$\CF$.
\end{theorem}

Note that the argument in Lemma~\ref{notpolycf} 
does not work for $\iota(\FIM_k)$, since the set 
$\{(x_1^{m_1} \olx_1^{m_1} \ldots x_k^{m_k} \olx_k^{m_k},
x_1^{m_1} \olx_1^{m_1} \ldots x_k^{m_k} \olx_k^{m_k}) \mid m_i\in \N_0\}$
is context-free.  Writing the idempotent on the second tape in reverse 
(that is, starting with $x_k^{m_k}\olx^{m_k}$)
does not help, as the resulting language is still $2$-$\CF$.
It appears likely that the intersection of $\iota(\FIM_k)$ with the following 
$\CF$ language would not be $(k-1)$-$\CF$:
\[\{(x_1^{l_1} \olx_1^{l_1} \ldots x_k^{l_k} \olx_k^{l_k}
x_1^{m_1} \olx_1^{m_1} \ldots x_k^{m_k} \olx_k^{m_k},
x_1^{n_1} \olx_1^{n_1} \ldots x_k^{n_k} \olx_k^{n_k})
\mid l_i,m_i,n_i\in \N_0\},\]
but proving this would require delicate arguments about 
intersections of stratified semilinear sets beyond the scope 
of this paper.

The author conjectures that neither version of the word problem 
for $\FIM_k$, $k\geq 2$ is indexed.  While a nested stack 
automaton can easily be used to store a Munn tree, there 
appears to be no way to check while reading a word 
$w\in X^*$ that the path traced out by $w$ visits every 
leaf of the stored Munn tree.

\end{document}